\documentclass[11pt, a4paper]{amsart}
\usepackage{a4}
\usepackage{amssymb}
\usepackage{amsmath}
\usepackage{amsthm}
\usepackage{amstext}
\usepackage{amscd}
\usepackage{latexsym}
\usepackage{mathabx}
\usepackage{graphics}
\usepackage{color}
\usepackage[all]{xy}
\usepackage{verbatim}

\usepackage[colorlinks, pagebackref]{hyperref}
\hypersetup{
  colorlinks=true,
  citecolor=blue,
  linkcolor=blue,
  urlcolor=blue}

\makeatletter
\def\@tocline#1#2#3#4#5#6#7{\relax
  \ifnum #1>\c@tocdepth 
  \else
    \par \addpenalty\@secpenalty\addvspace{#2}%
    \begingroup \hyphenpenalty\@M
    \@ifempty{#4}{%
      \@tempdima\csname r@tocindent\number#1\endcsname\relax
    }{%
      \@tempdima#4\relax
    }%
    \parindent\z@ \leftskip#3\relax \advance\leftskip\@tempdima\relax
    \rightskip\@pnumwidth plus4em \parfillskip-\@pnumwidth
    #5\leavevmode\hskip-\@tempdima
      \ifcase #1
       \or\or \hskip 2em \or \hskip 2em \else \hskip 3em \fi%
      #6\nobreak\relax
    \dotfill\hbox to\@pnumwidth{\@tocpagenum{#7}}\par
    \nobreak
    \endgroup
  \fi}
\makeatother

\newtheorem{intro-thm}{Theorem}[]

\theoremstyle{plain}
\newtheorem{thm}{Theorem}[section]

\newtheorem{lemma}[thm]{Lemma}
\newtheorem{lem}[thm]{Lemma}
\newtheorem{cor}[thm]{Corollary}

\newtheorem{prop}[thm]{Proposition}
\newtheorem{conj}[thm]{Conjecture}
\theoremstyle{definition}

\newtheorem{notation}[thm]{Notation}
\newtheorem{construction}[thm]{Construction}
\newtheorem{convention}[thm]{Convention}
\newtheorem{rmk}[thm]{Remark}

\newtheorem{defn}[thm]{Definition}
\newtheorem{ex}[thm]{Example}

\newcommand{\dlim}{\mathop{\varinjlim}\limits}  

\newcommand{\Dim}{{\rm dim}}

\newcommand{\Hom}{{\rm Hom}}

\newcommand{\Spec}{{\rm Spec \,}}
\newcommand{\Sing}{{\rm Sing}}

\newcommand{\sE}{{\mathcal E}}
\newcommand{\sF}{{\mathcal F}}

\newcommand{\sH}{{\mathcal H}}

\newcommand{\sL}{{\mathcal L}}

\newcommand{\sO}{{\mathcal O}}

\newcommand{\sS}{{\mathcal S}}

\newcommand{\sX}{{\mathcal X}}

\newcommand{\A}{{\mathbb A}}

\newcommand{\C}{{\mathbb C}}

\newcommand{\G}{{\mathbb G}}

\newcommand{\N}{{\mathbb N}}

\renewcommand{\P}{{\mathbb P}}

\newcommand{\Sh}{\sS h}
\newcommand{\deltaop}{\Delta^{op}(\sS h(Sm/{k}))}
\newcommand{\pdeltaop}{\Delta^{op}(P\sS h(Sm/{k}))}
\newcommand{\psh}{\pi_0}

\def\<{\langle}
\def\>{\rangle} 
\def\-{\overline} 
\def\~{\widetilde}
\def\^{\widehat}

\input{xy}
\xyoption{all}
\begin{document}

\title{$\A^1$-connected components of schemes}

\author{Chetan Balwe}

\address{School of Mathematics, Tata Institute of Fundamental
Research, Homi Bhabha Road, Bombay 400005, India}

\email{cbalwe@math.tifr.res.in}

\author{Amit Hogadi}

\address{School of Mathematics, Tata Institute of Fundamental
Research, Homi Bhabha Road, Bombay 400005, India}

\email{amit@math.tifr.res.in}

\author{Anand Sawant}

\address{School of Mathematics, Tata Institute of Fundamental
Research, Homi Bhabha Road, Bombay 400005, India}

\email{anands@math.tifr.res.in}

\thanks{Anand Sawant was supported by the Council of Scientific and Industrial Research, India under the Shyama Prasad Mukherjee Fellowship SPM-07/858(0096)/2011-EMR-I}


\date{}

\begin{abstract}
A conjecture of Morel asserts that the sheaf $\pi_0^{\A^1}(\sX)$ of $\A^1$-connected components of a simplicial sheaf $\sX$ is $\A^1$-invariant.  A conjecture of Asok and Morel asserts that the sheaves of $\A^1$-connected components of smooth schemes over a field coincide with the sheaves of their $\A^1$-chain-connected components.  Another conjecture of Asok and Morel states that the sheaf of $\A^1$-connected components is a birational invariant of smooth proper schemes.  In this article, we exhibit examples of schemes for which conjectures of Asok-Morel fail to hold and whose $Sing_*$ is not $\A^1$-local.  We also give equivalent conditions for Morel's conjecture to hold and obtain an explicit conjectural description of $\pi_0^{\A^1}(\sX)$.  A method suggested by these results is then used to prove Morel's conjecture for non-uniruled surfaces over a field.
\end{abstract}

\maketitle

{\small \tableofcontents}
\section{Introduction}

$\A^1$-homotopy theory, developed by Morel and Voevodsky \cite{Morel-Voevodsky} in the 1990's, is a \emph{homotopy theory} for schemes in which the affine line $\A^1$ plays the role of the unit interval in usual homotopy theory.  In the last few years, Morel has developed \emph{$\A^1$-algebraic topology} over a perfect field \cite{Morel} and obtained analogues of many important theorems in classical algebraic topology in the $\A^1$-homotopic realm.  In $\A^1$-homotopy theory, one needs to suitably enlarge the category of smooth schemes over a field in order to be able to perform various categorical constructions.  For various technical reasons, in $\A^1$-homotopy theory, one enlarges the category of smooth schemes over a field to the category of simplicial presheaves/sheaves of sets over the site of smooth schemes over a field with the Nisnevich topology.  By inverting \emph{$\A^1$-weak equivalences}, one obtains the \emph{$\A^1$-homotopy category}, defined in \cite{Morel-Voevodsky} by Morel-Voevodsky.  Analogous to algebraic topology, given a simplicial presheaf/sheaf $\sX$, one then studies the $\A^1$-homotopy sheaves $\pi_0^{\A^1}(\sX)$ and $\pi_i^{\A^1}(\sX, x)$, where $i\geq 1$ and $x \in \sX(\Spec k)$ is a basepoint. Throughout this paper, we will work over a perfect base field $k$ and denote the unpointed $\A^1$-homotopy category by $\sH(k)$.   

A sheaf of sets $\sF$ on the Nisnevich site $Sm/k$ of smooth schemes over $k$ is said to be \emph{$\A^1$-invariant} if the projection map $U \times \A^1 \to U$ induces a bijection $\sF(U) \to \sF(U \times \A^1)$, for all smooth schemes $U$.  Morel has proved the $\A^1$-invariance of the higher $\A^1$-homotopy sheaves of groups $\pi_i^{\A^1}(\sX,x)$, for $i \geq 1$, for any simplicial sheaf of sets $\sX$ over smooth schemes over a perfect field $k$ (see \cite[Theorem 6.1 and Corollary 6.2]{Morel}).  Morel has conjectured that the sheaf of $\A^1$-connected components $\pi_0^{\A^1}(\sX)$ of a simplicial sheaf $\sX$ is $\A^1$-invariant.  Morel's conjecture has been proved for motivic $H$-groups and homogeneous spaces for motivic $H$-groups by Choudhury in \cite{Choudhury}.  Any scheme $X$ over $k$ can be viewed as an object of the $\A^1$-homotopy category by means of the Nisnevich sheaf on $Sm/k$ given by the functor of points of $X$.  The $\A^1$-invariance of $\pi_0^{\A^1}(X)$ is easily seen to hold when $X$ is a scheme of dimension $\leq 1$ or an \emph{$\A^1$-rigid scheme} (for example, an abelian variety; see Definition \ref{definition A1 rigid}) or an \emph{$\A^1$-connected scheme} (such as the affine line $\A^1$).  It is also known to hold for smooth toric varieties (see Wendt \cite[Lemma 4.2 and Lemma 4.4]{Wendt}).  We now briefly describe the main results of this paper.  

\begin{intro-thm}[Theorem \ref{theorem lim S^n}, Corollary \ref{a1-invariance-candidate}] 
\label{Intro-theorem lim S^n}
If $\sF$ is a Nisnevich sheaf of sets on $Sm/k$, considered as a simplicially constant sheaf in $\bigtriangleup^{op}Sh(Sm/k)$.  Let $\mathcal S(\sF)$ denote the sheaf of $\A^1$-chain connected components of $\sF$ (see Definition \ref{definition A1 chain connected components}).  Then the sheaf $\sL(\sF):= \underset{n}{\varinjlim}~ \mathcal S^n(\sF)$ is $\A^1$-invariant.  Furthermore, if Morel's conjecture (Conjecture \ref{a1invariance}) is true, then the canonical map $$\pi_0^{\A^1}(\sF) \to \sL(\sF)$$ is an isomorphism.
\end{intro-thm}

The sheaf $\sS(X)$ (referred to as $\pi_0^{ch}(X)$ by Asok and Morel) of $\A^1$-chain-connected components of a smooth scheme $X$ over a field $k$ has been studied in the work of Asok and Morel \cite{Asok-Morel}, where they provide explicit connections between the notions of $\A^1$-connectedness and chain-connectedness and various rationality and near-rationality properties of the scheme $X$.  They have conjectured that the canonical map $\pi_0^{ch}(X) \to \pi_0^{\A^1}(X)$ is an isomorphism, for all smooth schemes $X$ (see \cite[Conjecture 2.2.8]{Asok-Morel}).  In support of this, they have shown that for a proper scheme $X$, the sections of the two sheaves $\pi_0^{ch}(X)$ and $\pi_0^{\A^1}(X)$ agree over all finitely generated, separable field extensions of $k$.  This is proved in \cite[Theorem 2.4.3 and Section 6]{Asok-Morel} by constructing another sheaf $\pi_0^{b\A^1}(X)$ of \emph{birational $\A^1$-connected components} of $X$ and by showing that its sections over finitely generated separable field extensions of $k$ are 
the same as those of $\pi_0^{ch}(X)$.  Asok and Morel have also conjectured that the natural map $\pi_0^{\A^1}(X) \to \pi_0^{b\A^1}(X)$ is an isomorphism, for all proper schemes $X$ of finite type over a field (see \cite[Conjecture 6.2.7]{Asok-Morel}).  

We obtain a refinement of the result of Asok and Morel \cite[Theorem 2.4.3]{Asok-Morel} about sections of $\pi_0^{\A^1}(X)$ over field extensions of $k$ using a method suggested by Theorem \ref{Intro-theorem lim S^n}.  Our proof gives the result of Asok and Morel by a direct geometric argument and completely avoids the use of $\pi_0^{b\A^1}(X)$. 

\begin{intro-thm}[Theorem \ref{theorem field case}, Corollary \ref{corollary field case}] 
\label{Intro-theorem field case}
Let $X$ be a proper scheme over $k$ and let $L$ be a finitely generated field extension of $k$.  For every positive integer $n$, we have $\sS(X)(\Spec L) = \sS^n(X)(\Spec L)$.  Consequently, the canonical epimorphism $\sS(X)(\Spec L) \to \pi_0^{\A^1}(X)(\Spec L)$ is an isomorphism. 
\end{intro-thm}

Theorem \ref{Intro-theorem lim S^n} suggests that the conjectures of Morel and of Asok-Morel could be proved for a scheme $X$ if we prove that $\sS(X) = \sS^2(X)$. This observation allows one to bring in geometric methods to study this question.  In this paper, we have taken this approach to prove the conjectures of Morel and of Asok-Morel for (possibly singular) proper non-uniruled surfaces (see Theorem \ref{theorem non-uniruled}).  However, we also present an example of a smooth, proper variety $X$ (of dimension $>2$, over $\mathbb C$) for which $\sS(X) \neq \sS^2(X)$.  For this variety $X$, we show that the morphism $\sS(X) \to \pi_0^{\A^1}(X)$ is not an isomorphism (disproving the conjecture of Asok-Morel \cite[Conjecture 2.2.8]{Asok-Morel}).  We also show that $Sing_*(X)$ is not $\A^1$-local, thus answering a question raised in \cite[Remark 2.2.9]{Asok-Morel}.  We also exhibit an example of a smooth proper scheme $X$ (Example \ref{example pi_0 not birational}) for which the conjecture by Asok-Morel about the birationality of $\pi_0^{\A^1}(X)$ fails to hold.

Section \ref{Section lim S^n} begins with some preliminaries and generalities on $\A^1$-homotopy theory.  We then give a proof of Theorem \ref{Intro-theorem lim S^n} and obtain equivalent characterizations of Morel's conjecture about $\A^1$-invariance of $\pi_0^{\A^1}$ of a simplicial sheaf and an explicit conjectural description of $\pi_0^{\A^1}(\sX)$.  In Section \ref{Section Asok-Morel}, we study the sheaves of $\A^1$-connected and chain-connected components of a scheme and give a proof of Theorem \ref{Intro-theorem field case} and prove the conjectures of Morel and of Asok-Morel for non-uniruled surfaces.  Finally, in Section \ref{section counterexamples}, we give an examples of schemes $X$ for which the conjectures of Asok-Morel about birationality of $\pi_0^{\A^1}(X)$ and equivalence of $\A^1$-connected and chain-connected components fail to hold.

\section{\texorpdfstring{$\A^1$}{A1}-connected components of a simplicial sheaf} 
\label{Section lim S^n}

\subsection{Basic definitions and generalities}
\label{subsection generalities}

We will always work over a base field $k$. Let $Sm/k$ denote the category of smooth finite type schemes over $k$ with Nisnevich topology. We let $\pdeltaop$ (resp. $\deltaop$) denote the category of simplicial presheaves (resp. sheaves) of sets on $Sm/k$. All presheaves or sheaves will also be considered as constant simplicial objects. There is an adjunction 
\[ 
\pdeltaop \overset{a_{Nis}}{\underset{i}{\rightleftarrows}} \deltaop , 
\]
\noindent where $i$ denotes the inclusion and $a_{Nis}$ denotes the Nisnevich sheafification. Both the categories have a model category structure, called the \emph{Nisnevich local model structure} (see \cite{Morel-Voevodsky} and \cite[Appendix B]{Jardine}), where weak equivalences are local (stalkwise) weak equivalences and cofibrations are monomorphisms.  The corresponding homotopy category of $\deltaop$ is called the \emph{simplicial homotopy category} and is denoted by $\mathcal H_s(k)$.  Moreover, the above adjunction is a Quillen equivalence.  

The left Bousfield localization of this model structure on $\bigtriangleup^{op}P\Sh(Sm/k)$ with respect to the collection of projection maps $\mathcal{X} \times \mathbb{A}^1 \to \mathcal{X}$ is called the \emph{$\mathbb{A}^1$-model structure}.  The resulting homotopy category is called the \emph{$\A^1$-homotopy category} and is denoted by $\mathcal{H}(k)$. 

There exists an \emph{$\A^1$-fibrant replacement functor} 
\[
L_{\mathbb{A}^1} : \pdeltaop \to \pdeltaop
\] 
such that for any space $\mathcal{X}$, the object $L_{\A^1}(\mathcal{X})$ is an $\mathbb{A}^1$-fibrant object. Moreover, there exists a canonical morphism $\mathcal{X} \to L_{\A^1}(\mathcal{X})$  which is an $\mathbb{A}^1$-weak equivalence.  A simplicial sheaf on $Sm/k$ is $\mathbb{A}^1$-fibrant if and only if it is simplicially fibrant (that is, fibrant in the Nisnevich local model structure) and $\mathbb{A}^1$-local (\cite[\textsection 2,  Proposition 3.19]{Morel-Voevodsky}).  See \cite[\textsection 2, Theorem 1.66 and p. 107]{Morel-Voevodsky} for the definition of the $\A^1$-fibrant replacement functor and more details. 

With the above notation, it is easy to see that $i$ preserves fibrations, $\A^1$-fibrations and $\A^1$-local objects,  and $a_{Nis}$ preserves cofibrations and $\A^1$-weak equivalences. 

\begin{defn}
For any simplicial presheaf (or a sheaf) $\sX$ on $Sm/k$, we define $\pi_0^s(\sX)$ to be the Nisnevich sheafification of the presheaf of sets $\psh(\sX)$ given by
\[
\psh(\sX)(U) := \Hom_{\mathcal H_s(k)}(U, \sX).
\] 
\end{defn}

\begin{notation}\label{notation1}
If we have two sets $R,S$ with maps $f,g:R \to S$, we denote by $\dfrac{S}{_gR_f}$, the quotient of $S$ by the equivalence relation generated by declaring $f(t)\sim g(t)$ for all $t\in R$. In this notation, for any simplicial presheaf $\sX$,
$$ \psh(\sX)(U) = \frac{\sX_0(U)}{_{d_0}\sX_1(U)_{d_1}},$$
where $d_0,d_1: \sX_1(U) \to \sX_0(U)$ are the face maps in the simplicial set $\sX_{\bullet}(U)$.
\end{notation}

\begin{defn} \label{definition a1 connected components}
The sheaf of \emph{$\A^1$-connected components} of a simplicial sheaf $\sX \in \deltaop$ is defined to be $$\pi_0^{\A^1}(\sX) := \pi_0^s(L_{\A^1}(\sX)).$$
\end{defn}

In other words, for $\sX \in \deltaop$, $\pi_0^{\A^1}(\sX)$ is the sheafification of the presheaf $$U \in Sm/k ~ \mapsto ~ \Hom_{\mathcal H(k)}(U, \sX).$$

\begin{defn} \label{definition a1 invariance}
A presheaf $\sF$ of sets on $Sm/k$ is said to be {\em $\A^1$-invariant} if for every $U$, the map $ \sF(U)\to \sF(U\times \A^1)$ induced by the projection $U\times \A^1 \to U$, is a bijection. 
\end{defn}

It is easy to see that the presheaf 
$$ U \mapsto \psh(L_{\A^1}(\sX)(U))$$ 
is $\A^1$-invariant. A conjecture of Morel (see \cite[1.12]{Morel}) asserts that its Nisnevich sheafification $\pi_0^{\A^1}(\sX)$ remains $\A^1$-invariant.

\begin{conj}[Morel]\label{a1invariance}
For any simplicial sheaf $\sX$, $\pi_0^{\A^1}(\sX)$ is $\A^1$-invariant.
\end{conj}

The following example shows that sheafification in Nisnevich topology can destroy $\A^1$-invariance of a presheaf, in general. 

\begin{ex}
Consider the presheaf $\sF$ on $Sm/k$ whose sections $\sF(U)$ are defined to be the set of $k$-morphisms from $U \to \A^1_{k}$ which factor through a proper open subset of $\A^1_{k}$. One can observe that $\sF$ is $\A^1$-invariant, but its Nisnevich sheafification, which is the sheaf represented by $\A^1_{k}$, is clearly not $\A^1$-invariant.  
\end{ex}

In \cite[Theorem 4.18]{Choudhury},  Morel's conjecture was proved in the special case when $\sX$ is a motivic $H$-group or a homogeneous space for a motivic $H$-group. As observed in \cite{Choudhury}, such simplicial sheaves $\sX$ have a special property - for $U \subset Y$, an inclusion of dense open subscheme, $\pi_0^s(\sX)(Y)\to \pi_0^s(\sX)(U)$ is injective (see \cite[Corollary 4.17]{Choudhury}). Unfortunately, this injectivity property fails for general simplicial sheaves $\sX$, as shown by the following example.

\begin{ex}\label{nonseparated}
Let $U$ be a non-empty proper subscheme of an abelian variety $A$ over $k$. Let $X$ be the non-separated scheme obtained by gluing two copies of $A$ along $U$.  Let $X$ continue to denote the corresponding constant simplicial sheaf in $\deltaop$.  Clearly, $X(A)\to X(U)$ is not injective.  However, one can verify that the sheaf represented by $X$ is $\A^1$-invariant  and hence, $X=\pi_0^{\A^1}(X)$. 
\end{ex}

For any simplicial sheaf $\sX$, it is easy to see from the definitions that the sheaf $\pi_0^{\A^1}(\sX)$ has the following universal property.

\begin{lem} \label{universal property pi_0A1}
Let $\sX \in \deltaop$ be a simplicial sheaf and $\sF$ be an $\A^1$-invariant sheaf of sets on $Sm/k$.  Then any morphism $\sX \to \sF$ factors uniquely through the canonical morphism $\sX \to \pi_0^{\A^1}(\sX)$.  
\end{lem}
\begin{proof}
We have a commutative diagram 
\[
\xymatrix{
\sX \ar[d] \ar[r] & \sF \ar[d] \\
L_{\A^1}(\sX) \ar[r] & L_{\A^1}(\sF)
.} 
\]
\noindent Since $\sF$ is $\A^1$-invariant, the map $\sF \to L_{\A^1}(\sF)$ is an isomorphism.  Hence, the map $\sX \to \sF$ factors through $L_{\A^1}(\sX)$. The map $L_{\A^1}(\sX) \to \sF$ factors through $\pi_0^{\A^1}(\sX) = \pi_0^s(L_{\A^1}(\sX))$ since $\sF$ is of simplicial dimension $0$.  
\end{proof}

We now recall the  construction of the functor $Sing_*$ as given in \cite[p. 87]{Morel-Voevodsky}. Let $\Delta_{\bullet}$ denote the cosimplicial sheaf where 
$$ \Delta_n  = \Spec\left(\frac{k[x_0,...,x_n]}{(\sum_ix_i=1)}\right)$$
with obvious coface and codegeneracy maps motivated from those on topological simplices. For any simplicial presheaf (or a sheaf) $\sX$, define $Sing_*(\sX)$ to be the  diagonal of the bisimplicial presheaf $\underline{\Hom}(\Delta_{\bullet},\sX)$, where $\underline{\Hom}$ denotes the internal Hom.  Concretely,
$$Sing_*(\sX)_n = \underline{\Hom}(\Delta_n,\sX_n).$$ 
Note that if $\sX$ is a sheaf, then so is $Sing_*(\sX)$. There is a functorial morphism $\sX\to Sing_*(\sX)$ induced by $\sX_n(U) \to \sX_n(U\times\Delta_n)$, which is an $\A^1$-weak equivalence.  The functor $Sing_*$ takes $\A^1$-fibrant objects to $\A^1$-fibrant objects and takes $\A^1$-homotopic maps to simplicially homotopic maps.

\begin{defn} \label{definition A1 chain connected components}
Let $\sF$ be a Nisnevich sheaf of sets on $Sm/k$, considered as a simplicially constant sheaf in $\bigtriangleup^{op}Sh(Sm/k)$.  Define $\sS(\sF)$ to be the sheaf associated to the presheaf $\sS^{pre}$ given by
\[
\sS^{pre}(U) := \frac{\sF(U)}{_{\sigma_0}\sF(U\times\A^1)_{\sigma_1}},
\] 
\noindent for $U \in Sm/k$, where $\sigma_0, \sigma_1 : \sF(U \times \A^1) \to \sF(U)$ are the maps induced by the $0$- and $1$-sections $U \to U \times \A^1$, respectively (see Notation \ref{notation1}).  Note that, even if $\sF$ is a sheaf, $\sS^{pre}(\sF)$ need not be a sheaf.  For any $n > 1$, we inductively define the sheaves 
\[
\sS^n(\sF) := \sS(\sS^{n-1}(\sF)). 
\]
\end{defn}

For any sheaf $\sF$, there exists a canonical epimorphism $\sF \to \sS(\sF)$.  This gives a canonical epimorphism $$\sF \to \underset{n}{\varinjlim}~ \sS^n(\sF).$$ 

\begin{rmk} \label{remark-S-Sing}
It is clear from the definition of $Sing_*(\sF)$ that
$$ \psh(Sing_*(\sF))(U) = \frac{\sF(U)}{_{\sigma_0}\sF(U\times\A^1)_{\sigma_1}}.$$
Thus, $\sS(\sF) = \pi_0^s(Sing_*(\sF))$, for any sheaf $\sF$. 
\end{rmk}

\begin{rmk}
Let $X$ be a smooth scheme over $k$ and view it as a Nisnevich sheaf of sets over $Sm/k$.  The sheaf $\sS(X)$ defined above (Definition \ref{definition A1 chain connected components}) is none other than the sheaf $\pi_0^{ch}(X)$ of \emph{$\A^1$-chain connected components} of $X$ introduced by Asok and Morel \cite[Definition 2.2.4]{Asok-Morel}.  We use $\sS$ instead of $\pi_0^{ch}$ in this paper only for typographical reasons.  
\end{rmk}

For smooth schemes over a field $k$, Asok and Morel have conjectured the following (see \cite[Conjecture 2.2.8]{Asok-Morel}):

\begin{conj} \label{conjecture Asok-Morel}
For any smooth scheme $X$ over a field $k$, the natural epimorphism $\sS(X) \to \pi_0^{\A^1}(X)$ is an isomorphism. 
\end{conj}

Asok and Morel also mention that this would be true if one proves that $Sing_*(X)$ is $\A^1$-local, for any smooth scheme $X$ over $k$ (see \cite[Remark 2.2.9]{Asok-Morel}).  However, in Section \ref{section counterexamples}, we give examples of schemes for which this property and Conjecture \ref{conjecture Asok-Morel} fail to hold.

\subsection{Remarks on Morel's conjecture on \texorpdfstring{$\A^1$}{A1}-invariance of \texorpdfstring{$\pi_0^{\A^1}(\sX)$}{pi0A1(X)}}

In this subsection, we use the sheaf of $\A^1$-chain connected components of a (Nisnevich) sheaf $\sF$ on $Sm/k$ to construct a sheaf $\sL(\sF)$ that is closely related to $\pi_0^{\A^1}(\sF)$ and prove Theorem \ref{Intro-theorem lim S^n} stated in the introduction.

\begin{thm}\label{theorem lim S^n}
For any sheaf of sets $\sF$ on $Sm/k$, the sheaf $$\sL(\sF) := \underset{n}{\varinjlim}~ S^n(\sF)$$ is $\A^1$-invariant.
\end{thm}
\begin{proof}
It is enough to show that for every smooth $k$-scheme $U$, the map
$$\sL(\sF)(U)\to \sL(\sF)(U\times\A^1)$$ is surjective, since it is already injective because the projection map $U \times \A^1 \to U$ admits a section.
Let $t\in \sL(\sF)(U\times \A^1)$.  Since the sheaf $\sL(\sF)$ is a filtered colimit, we have for each $U \in Sm/k$
$$ \sL(\sF)(U) = \dlim_{n \in \N} \sS^n(\sF)(U).$$
Thus, $t$ is represented by an element $t_n$ of $\sS^n(\sF)(U\times \A^1)$ for some $n$. 
We will show that $t$ is contained in the image of $\sL(\sF)(U)$ by showing that the image of $t_n$ in $\sS^{n+1}(U \times \A^1)$ is contained in the image of $\sS^{n+1}(\sF)(U)$.  Let $m$ denote the \emph{multiplication} map 
$$ U\times \A^1 \times \A^1 \to U\times \A^1  \ \ ; \ \ (u,x,y) \mapsto (u,xy)$$
Consider the element $$m^*(t_n) \in \sS^n(\sF)(U\times \A^1\times \A^1).$$
Then $$\sigma_1^*\circ m^* (t_n) = t_n, \ \ \text{and} \ \ \sigma_0^*\circ m^*(t_n) = p^*\circ \sigma_0^*(t_n)$$
where $U\times \A^1 \stackrel{p}{\to} U$ is the projection. Thus, the image of $t_n$ in $\sS^{n+1}(\sF)(U\times \A^1)$ is contained in the image of the map $$\sS^{n+1}(\sF)(U) \to \sS^{n+1}(\sF)(U\times \A^1).$$ This proves the result.
\end{proof}

\begin{rmk}\label{remark factorization through pi_0-a1}
In view of Theorem \ref{theorem lim S^n}, the canonical map $\sF \to \sL(\sF)$ uniquely factors through the canonical map $\sF \to \pi_0^{\A^1}(\sF)$.   
\end{rmk}

\begin{rmk}\label{remark universal property L(F)}
Note that $\sL(\sF)$ satisfies the following universal property: any map from a sheaf $\sF$ to an $\A^1$-invariant sheaf uniquely factors through the canonical map $\sF \to \sL(\sF)$.  Recall that $\pi_0^{\A^1}$ also satisfies the same universal property (Lemma \ref{universal property pi_0A1}) but is not known to be $\A^1$-invariant in general.
\end{rmk}

We now see how Theorem \ref{theorem lim S^n} gives us a reformulation of Morel's conjecture on the $\A^1$-invariance of the $\A^1$-connected components sheaf $\pi_0^{\A^1}$.

\begin{lem}\label{a1-invariance-equiv-sheaves}
Let $\sF$ be a sheaf of sets on $Sm/k$, considered as an object of $\bigtriangleup^{op}\Sh(Sm/k)$.  The following are equivalent:\\
\noindent $(1)$ The sheaf $\pi_0^{\A^1}(\sF)$ is $\A^1$-invariant. \\
\noindent $(2)$ The canonical map $$\pi_0^{\A^1}(\sF) \to \sL(\sF)$$ admits a retract. 
\end{lem}
\begin{proof}
\noindent $(1) \Rightarrow (2)$ is straightforward by Remark \ref{remark universal property L(F)} and the universal property of $\pi_0^{\A^1}$ (Lemma \ref{universal property pi_0A1}). $(2) \Rightarrow (1)$ follows since $\sL(\sF)$ is $\A^1$-invariant (Theorem \ref{theorem lim S^n}) and since a retract of an $\A^1$-invariant sheaf is $\A^1$-invariant.
\end{proof}

\begin{prop}\label{a1-invariance-equiv-spaces}
The following are equivalent:\\
\noindent $(1)$ The sheaf $\pi_0^{\A^1}(\sX)$ is $\A^1$-invariant, for all $\sX \in \bigtriangleup^{op}\Sh(Sm/k)$. \\
\noindent $(2)$ The canonical map $\pi_0^{\A^1}(\sX) \to \pi_0^{\A^1}(\pi_0^s(\sX))$ is an isomorphism, for all $\sX \in \bigtriangleup^{op}\Sh(Sm/k)$.
\end{prop}
\begin{proof}
\noindent $(1) \Rightarrow (2):$  Consider the following commutative diagram with the natural morphisms 
\[
\xymatrix{
\sX \ar[r] & \pi_0^s(\sX) \ar[r] \ar[d] & \pi_0^{\A^1}(\sX) \ar[dl] \\
           & \pi_0^{\A^1}(\pi_0^s(\sX)) &   }
\]
\noindent All the morphisms in the above diagram are epimorphisms (the right horizontal map is an epimorphism by the unstable $\A^1$-connectivity theorem \cite[\textsection 2, Corollary 3.22]{Morel-Voevodsky}).  Since $\pi_0^{\A^1}(\sX)$ is $\A^1$-invariant by assumption, the map $\pi_0^s(\sX) \to \pi_0^{\A^1}(\sX)$ has to uniquely factor through $\pi_0^{\A^1}(\pi_0^s(\sX))$, by the universal property of $\pi_0^{\A^1}(\pi_0^s(\sX))$.  This gives an inverse to the map $\pi_0^{\A^1}(\sX) \to \pi_0^{\A^1}(\pi_0^s(\sX))$, by uniqueness. \\
\noindent $(2) \Rightarrow (1):$ Let $\sX$ be a space and let $\sF := \pi_0^s(\sX)$ be the sheaf of (simplicial) connected components of $\sX$.  We are given that $\pi_0^{\A^1}(\sX) \to \pi_0^{\A^1}(\sF)$ is an isomorphism.  We will prove that $\pi_0^{\A^1}(\sF)$ is $\A^1$-invariant.  We have a natural map $$\sS(\sF) = \pi_0^s(Sing_*(\sF)) \to \pi_0^{\A^1}(Sing_*(\sF)) \simeq \pi_0^{\A^1}(\sF).$$  By Remark \ref{remark-S-Sing}, we have 
\[
\sS^2(\sF) =  \sS (\pi_0^s(Sing_*(\sF))) = \pi_0^s(Sing_* \pi_0^s(Sing_*(\sF))), 
\]
\noindent whence we have a map $\sS^2(\sF) \to \pi_0^{\A^1}(Sing_* \pi_0^s(Sing_*(\sF)))$.  But hypothesis $(2)$ implies that
\[
\pi_0^{\A^1}(Sing_* \pi_0^s(Sing_*(\sF))) \simeq \pi_0^{\A^1}(\pi_0^s(Sing_*(\sF))) \simeq \pi_0^{\A^1}(Sing_*(\sF)) \simeq \pi_0^{\A^1}(\sF).
\]
\noindent  The composition of these maps gives a natural map $\sS^2(\sF) \to \pi_0^{\A^1}(\sF)$, which makes the following diagram commute, where the maps are the ones defined above.
\[
\xymatrix{
\sS(\sF) \ar[dr] \ar[rr] &   & \sS^2(\sF) \ar[dl]\\
  {}           & \pi_0^{\A^1}(\sF) & {}
}
\]
\noindent Continuing in this way, we obtain a compatible family of maps $\sS^i({\sF}) \to \pi_0^{\A^1}(\sF)$, for each $i$, giving a retract $\sL(\sF) \to \pi_0^{\A^1}(\sF)$.  Lemma \ref{a1-invariance-equiv-sheaves} now implies that $\pi_0^{\A^1}(\sF)$ is $\A^1$-invariant, proving $(1)$.
\end{proof}

These results immediately give a description of the $\A^1$-connected components sheaf $\pi_0^{\A^1}(\sX)$ for a simplicial sheaf of sets $\sX$ on $Sm/k$, provided Morel's conjecture is true.

\begin{cor}\label{a1-invariance-candidate}
Suppose Conjecture \ref{a1invariance} is true. Let $\sX$ be a simplicial sheaf of sets on $Sm/k$ and let $\sF := \pi_0^s(\sX)$.  Then the canonical maps $$\pi_0^{\A^1}(\sX) \to \pi_0^{\A^1}(\sF) \to \sL(\sF)$$ are isomorphisms.
\end{cor}

\section{\texorpdfstring{$\A^1$}{A1}-connected/chain-connected components of schemes}
\label{Section Asok-Morel}

\subsection{Sections of \texorpdfstring{$\pi_0^{\A^1}(X)$}{pi0A1(X)} over finitely generated field extensions of \texorpdfstring{$k$}{k}}
\label{subsection field case}

The aim of this subsection is to give a proof of Theorem \ref{Intro-theorem field case} stated in the introduction.  The strategy is the one suggested by Theorem \ref{theorem lim S^n}: we shall prove that for a proper scheme $X$ over $k$, $\sS(X)(\Spec L) = \sS^2(X)(\Spec L)$, for every field extension $L$ of $k$. 

Let $\sF$ be any sheaf of sets on $Sm/k$. In order to address the question of $\A^1$-invariance of $\sS(\sF)$, we need to obtain an explicit description for the elements of the set $\sS(\sF)(U)$ where $U$ is a smooth scheme over $k$.  We will then specialize to the case when $\sF$ is represented by a proper scheme $X$ over $k$.

As before, for any scheme $U$ over $k$, we let $\sigma_0$ and $\sigma_1$ denote the morphisms $U \to U \times \A^1$ given by $u \mapsto (u,0)$ and $u \mapsto (u,1)$, respectively. 

\begin{defn}
\label{definition homotopy}
Let $\sF$ be a sheaf of sets over $Sm/k$ and let $U$ be a smooth scheme over $k$.
\begin{itemize}
\item[(1)] An \emph{$\A^1$-homotopy} of $U$ in $\sF$ is a morphism an element $h$ of $\sF(U \times \A^1_{k})$. We say that $t_1, t_2 \in \sF(U)$ are \emph{$\A^1$-homotopic} if there exists an $\A^1$-homotopy $h \in \sF(U \times \A^1)$ such that $\sigma_0^*(h) = t_1$ and $\sigma_1^*(h)= t_2)$.
\item[(2)] An \emph{$\A^1$-chain homotopy} of $U$ in $\sF$ is a finite sequence $h=(h_1, \ldots, h_n)$ where each $h_i$ is an $\A^1$-homotopy of $U$ in $\sF$ such that $\sigma_1^*(h_i) = \sigma_0^*(h_{i+1})$ for $1 \leq i \leq n-1$. We say that $t_1, t_2 \in \sF(U)$ are \emph{$\A^1$-chain homotopic} if there exists an $\A^1$-chain homotopy $h=(h_1, \ldots, h_n)$ such that $\sigma_0^*(h_1) = t_1$ and $\sigma_1^*(h_n) = t_2$. 
\end{itemize}
\end{defn}

Clearly, if $t_1$ and $t_2$ are $\A^1$-chain homotopic, they map to the same element of $\sS(\sF)(U)$. The converse is partially true - if $t_1, t_2 \in \sF(U)$ are such that they map to the same element in $\sS(\sF)(U)$, then there exists a Nisnevich cover $V \to U$ such that $t_1|_{V}$ and $t_2|_{V}$ are $\A^1$-chain homotopic. 

Now we explicitly describe the sections of $\sS(\sF)$. Let $t \in \sS(\sF)(U)$. Since $\sF \to \sS(\sF)$ is an epimorphism of sheaves, there exists a finite Nisnevich cover $V \to U$ such that $t|_{V}$ can be lifted to $s \in \sF(V)$. Let $pr_1,pr_2: V \times_{U} V \to V$ be the two projections. Then, since the two elements $pr_i^*(s)$ map to the same element in $\sS(\sF)(V \times_U V)$, there exists a finite Nisnevich cover $W \to V \times_U V$ such that $pr_1^*(s)|_{W}$ and $pr_2^*(s)|_{W}$ are $\A^1$-chain homotopic.  Conversely, given a finite cover $V \to U$ and an element $s$ of $\sF(V)$ such that $pr_1^*(s)$ and $pr_2^*(s)$ become $\A^1$-chain homotopic after restricting to a Nisnevich cover of $V \times_U V$, we can obtain a unique element $t$ of $\sS(\sF)(U)$. 

Applying the same argument to to the sections of $\sS(\sF)$ over $\A^1_U$, we are led to define the notion of an \emph{$\A^1$-ghost homotopy}.

\begin{defn}
\label{definition ghost homotopy}
Let $\sF$ be a sheaf of sets and let $U$ be a smooth scheme over $k$. An \emph{$\A^1$-ghost homotopy} of $U$ in $\sF$ consists of the data 
\[
\sH:=(V \to \A^1_U, W \to V \times_{\A^1_U} V, h, h^W)
\]
which is defined as follows:
\begin{itemize}
\item[(1)] $V \to \A^1_U$ is a Nisnevich cover of $\A^1_U$.
\item[(2)] $W \to V \times_{\A^1_U} V$ is a Nisnevich cover of $V \times_{\A^1_U} V$.  
\item[(3)] $h$ is a morphism $V \to \sF$.
\item[(4)] $h^W$ is an $\A^1$-chain homotopy connecting the two morphisms $W \to V\times_{\A^1_U} V  \overset{pr_i}{\to} V \to F$ where $pr_1$ and $pr_2$ are the projections $V \times_{\A^1_U} V \to V$. (Thus $h^W$ is a finite sequence $(h^{W}_1, \ldots )$ of $\A^1$-homotopies satisfying the appropriate conditions as given in Definition \ref{definition homotopy}.)
\end{itemize}
Let $t_1, t_2 \in \sF(U)$. We say that \emph{$\sH$ connects $t_1$ and $t_2$} (which are then said to be \emph{$\A^1$-ghost homotopic}) if the morphisms $\sigma_0, \sigma_1: U \to \A^1_U$ admit lifts $\~\sigma_0: U \to V$ and $\~\sigma_1: U \to V$ such that $h \circ \~\sigma_0 = t_1$ and $h \circ \~\sigma_1 = t_2$. 
\end{defn}

By the discussion above, a homotopy in $\sS(\sF)$ gives rise (non-uniquely) to a ghost homotopy in $\sF$. On the other hand, a ghost homotopy in $\sF$ gives rise to a unique homotopy in $S(\sF)$. Also, if $t_1$, $t_2$ are $\A^1$-ghost homotopic elements of $\sF(U)$, then it is clear that their images in $\sS(\sF)(U)$ are $\A^1$-homotopic. The converse is partially true -- if $\~{t_1}, \~{t_2} \in \sS(\sF)(U)$ are $\A^1$-homotopic, then they have preimages $t_1, t_2 \in \sF(U)$ which become $\A^1$-ghost homotopic over some Nisnevich cover of $U$. In general, if $\sH$ is an $\A^1$-ghost homotopy of $U$, it is possible that there may not exist two elements $t_1, t_2 \in \sF(U)$ which are connected by $\sH$. This is because the lifts $\~\sigma_0$ and $\~\sigma_1$ may not exist until we pass to a Nisnevich cover of $U$. However, if $U$ is a smooth Henselian local scheme over $k$, then the morphisms $\~\sigma_0$ and $\~\sigma_1$ do exist and thus there exist preimages of $\~{t_1}$ and $\~{t_2}$ which are $\A^1$-
ghost homotopic. 

From the above discussion, we obtain the following obvious lemma:
\begin{lemma}
\label{lemma strategy}
Let $\sF$ be a sheaf of sets over $Sm/k$. Then, $\sS(\sF) = \sS^2(\sF)$ if and only if for every smooth Henselian local scheme $U$, if $t_1, t_2 \in \sF(U)$ are $\A^1$-ghost homotopic, then they are also $\A^1$-chain homotopic. 
\end{lemma}

The rest of this subsection is devoted to the proof of Theorem \ref{Intro-theorem field case}.  We introduce the concept of \emph{almost proper sheaves} (see Definition \ref{definition-almost-proper} below), which we shall find helpful while handling the iterations of the functor $\sS$, for the purposes of this paper.

\begin{notation}
\label{notation-agree-on-points}
Let $Z$ be a variety over $k$ and let $x$ be a point of $Z$. Let $k(x) = \sO_{Z,x}/\mathfrak{m}_x$ be the residue field at $x$. The quotient homomorphism $\sO_{X,x} \to k(x)$ induces a morphism $\Spec(k(x)) \to \Spec(\sO_{Z,x}) \to Z$ which will also be denoted by $x: \Spec(k(x)) \to Z$. 

Let $\sF$ be a sheaf of sets and let $Z$ be a variety over $k$.  Let $s_1, s_2 \in \sF(Z)$.  We will write  $s_1 \sim_0 s_2$ if for any point $x$ of $Z$, the morphisms $s_1 \circ x, s_2 \circ x: \Spec(k(x)) \to \sF$ are identical.  
\end{notation}

\begin{rmk}
In Notation \ref{notation-agree-on-points}, the extension $k(x)/k$ may have positive transcendence degree. For instance, this may happen when $U$ is a smooth, irreducible scheme of dimension $\geq 1$ over $k$ and $x$ is the generic point of $U$. Then $\Spec(k(x))$ is viewed as an essentially smooth scheme over $k$ and the set $\sF(\Spec(k(x))$ is the direct limit $$\sF(\Spec(k(x)) := \varinjlim_{U^{\prime}} \sF(U^{\prime})$$ where $U^{\prime}$ varies over all the open subschemes of $U$. Thus any morphism $\Spec(k(x)) \to \sF$ can be represented by a morphism $U^{\prime} \to \sF$ where $U^{\prime}$ is some open subscheme of $U$. In the following discussion, we will also have to consider morphisms of the form $\A^1_{k(x)} \to \sF$. Such a morphism can be represented by a morphism $\A^1_{U^{\prime}} \to \sF$ where $U^{\prime}$ is an open subscheme of $U$.   
\end{rmk}

\begin{defn}
\label{definition-almost-proper}
Let $\sF$ be a sheaf of sets over $Sm/k$. We say that $\sF$ is \emph{almost proper} if the following two properties hold:
\begin{itemize}
 \item[(AP1)] Let $U$ be a smooth, irreducible variety of dimension $\leq 2$ and let $s: U \to \sF$ be any morphism. Then there exists a smooth projective variety $\-U$, a birational map $i: U \dashrightarrow \-U$ and a morphism $\-s: \-U \to \sF$ which ``extends $s$ on points''. By this we mean that  if $U^{\prime}$ is the largest open subscheme of $U$ such that $i$ is represented by a morphism $i^{\prime}: U^{\prime} \to \-U$, then we have $\-s \circ i^{\prime} \sim_0 s|_{U^{\prime}}$.
 \item[(AP2)] Let $U$ be a smooth, irreducible curve over $k$ and let $U^{\prime}$ be an open subscheme. Suppose we have two morphisms $s_1, s_2: U \to \sF$ such that $s_1|_{U^{\prime}} = s_2|_{U^{\prime}}$. Then $s_1 \sim_0 s_2$. 
\end{itemize}
\end{defn}

It is clear that a proper scheme over $k$ represents an almost proper sheaf. Thus, in order to prove Theorem \ref{Intro-theorem field case} in the introduction, it will suffice to prove that if $\sF$ is an almost proper sheaf, then $\sS(\sF)(K) = \sS^n(\sF)(K)$ for any finitely generated field extension $K/k$. This will be done in Theorem \ref{theorem field case} below.  

\begin{lemma}
\label{lemma-connect-points}
Let $\sF$ be an almost proper sheaf of sets on $Sm/k$. Let $U$ be a smooth curve, let $x: \Spec(k) \to U$ be a $k$-rational point on $U$ and let $U^{\prime}$ be the open subscheme $U \backslash \{x\}$ of $U$. Let $f,g: U \to \sF$ be such that the morphisms $f|_{U^{\prime}}$ and $g|_{U^{\prime}}$ are $\A^1$-chain homotopic. Then the morphisms $f \circ x,  g \circ x: \Spec(k) \to \sF$ are $\A^1$-chain homotopic.  
\end{lemma}
\begin{proof}
First we prove the result in the case when $f|_{U^{\prime}}$ and $g|_{U^{\prime}}$ are simply $\A^1$-homotopic, i.e. there exists a morphism $h: U^{\prime} \times \A^1 \to \sF$ such that $h|_{U^{\prime} \times \{0\}} = f|_{U^{\prime}}$ and $h|_{U^{\prime} \times \{1\}} = g|_{U^{\prime}}$. 

By (AP1), there exists a smooth proper surface $X$ and a proper, birational morphism $i: U' \times \A^1 \dashrightarrow X$ and a morphism $\-h: X \to \sF$ ``extending $h$ on points". As in Definition \ref{definition-almost-proper}, we elaborate on the precise sense in which $\-h$ extends $h$. Suppose that $W$ is the largest open subscheme of $U^{\prime} \times \A^1$ on which $i$ is defined. Then $\-h \circ i|_W \sim_0 h|_W$. Note that $(U^{\prime} \times \A^1) \backslash W$ is a closed subscheme of codimension $2$. Thus there exists an open subscheme $U'' \subset U'$ such that $U'' \times \A^1 \subset W$.  The curves $U'' \times \{0\}$ and $U'' \times \{1\}$, which are contained in $W$,  are mapped into $X$ by $i$. The restrictions of $i$ to these curves induce morphisms $t_0, t_1: \-U \to X$, where $\-U$ is a smooth compactification of $U$. Since $\-h \circ i|_W \sim_0 h|_W$, we see that $\-h \circ t_0|_{U''} \sim_0 f|_{U''}$ and $\-h \circ t_1|_{U''} \sim_0 g|_{U''}$.  This implies that there exists a Zariski open subset $U'''$ of $U''$ such that $\-h \circ t_0|_{U'''} = f|_{U'''}$ and $\-h \circ t_1|_{U'''} = g|_{U''}$. Thus, by (AP2), $\-h \circ t_0 \circ x = f \circ x$ and $\-h \circ t_1 \circ x = g \circ x$. Thus in order to show that $f \circ x$ and $g \circ x$ are $\A^1$-chain homotopic, it suffices to show that $t_0 \circ x$ and $t_1 \circ x$ are $\A^1$-chain homotopic in $X$. Since $X$ is birational to $\-U \times \P^1$, by resolution of indeterminacy (see, for example, \cite[Chapter II, 7.17.3]{Hartshorne}), there exists a smooth projective surface $\~X$ and birational proper morphisms $\~X \to \-U \times \P^1$ and $\~X \to X$.  
\[
\xymatrix{
& \~X \ar[dl] \ar[dr] & \\
\-U \times \P^1 & & X \\
& W \ar@{_(->}[ul] \ar[ur]^i&
}
\]
\noindent Let $\-t_0: \-U \to \-U \times \P^1$ be the map that identifies $\-U$ with the curve $\-U \times \{(0:1)\}$ and $\-t_1: \-U \to \-U \times \P^1$ be the map that identifies $\-U$ with the curve $\-U \times \{(1:1)\}$.  It is easy to see that there exist unique morphisms $\~t_0, \~t_1: \-U \to \~X$ such that $\~t_0$ is a lift of both $t_0$ and $\-t_0$ and $\~t_1$ is a lift of both $t_1$ and $\-t_1$.  This completes the proof of the result when $f|_{U^{\prime}}$ and $g|_{U^{\prime}}$ are simply $\A^1$-homotopic by \cite[Lemma 6.2.9]{Asok-Morel}, since $\~t_0 \circ x$ and $\~t_1 \circ x$ are $\A^1$-chain homotopic.

Now, suppose $f|_{U^{\prime}}$ and $g|_{U^{\prime}}$ are $\A^1$-chain homotopic. Thus, we have a sequence $ f_0 = f|_{U^{\prime}}, \ldots, f_n =g|_{U^{\prime}}$ where $f_i$ is $\A^1$-homotopic to $f_{i+1}$ for every $i$. By (AP1), for each $i$, there exists a $\~f_i: U \to \sF$ such that $\~f_i|_{U^{\prime}} = f_i$. Thus $\~f_i \circ x$ and $\~f_{i+1} \circ x$ are $\A^1$-chain homotopic for every $i$. By (AP2), $\~f_0 \circ x = f \circ x$ and $\~f_n \circ x = g \circ x$.  
\end{proof}

\begin{lemma}
\label{lemma-almost-proper}
Let $\sF$ be an almost proper sheaf. Then $\sS(\sF)$ is also an almost proper sheaf.
\end{lemma}
\begin{proof}
We first check condition (AP2). Thus, let $U$ be a smooth curve and let $U^{\prime} \subset U$ be an open subscheme. We assume that we have two morphisms $s_1, s_2: U \to \sS(\sF)$ such that $s_1|_{U^{\prime}} = s_2|_{U^{\prime}}$ and prove that $s_1 \sim_0 s_2$. Without any loss of generality, we may assume that $U \backslash U^{\prime}$ consists of a single closed point. By a change of base, if necessary, we may assume that this point is $k$-rational, that is, it is the image of a morphism $x: \Spec(k) \to U$. We need to prove that $s_1 \circ x = s_2 \circ x$. 

The morphism $\phi: \sF \to \sS(\sF)$ is an epimorphism of sheaves. Thus, there exists a smooth, irreducible curve $C$ and an \'etale morphism $p: C \to U$ such that $x$ can be lifted to a morphism $c_0: \Spec(k) \to C$, (i.e. $p \circ c_0 = x$) and such that the morphisms $p \circ s_1$ and $p \circ s_2$ can be lifted to $\sF$. In other words, there exist morphisms $f,g: C \to \sF$ such that $\phi \circ f =s_1 \circ p$ and $\phi \circ g = s_2 \circ p$. Since $s_1|_{U^{\prime}} = s_2|_{U^{\prime}}$, there exists a Nisnevich cover $q: V \to U^{\prime}$ such that $q \circ f$ and $q \circ g$ are $\A^1$-chain homotopic. Let $K$ denote the function field of $C$. Then the canonical morphism $\eta: \Spec(K) \to \sF$ lifts to $V$. Thus the morphisms $f \circ \eta$ and $g \circ \eta$ are $\A^1$-chain homotopic. Thus, there exists an open subscheme $C^{\prime} \subset C$ such that if $i: C^{\prime} \hookrightarrow C$ is the inclusion, the morphisms $f \circ i$ and $g \circ i$ are $\A^1$-chain homotopic. Applying Lemma \ref{lemma-connect-points}, we see that $f \circ c_0, g \circ c_0: \Spec(k) \to \sF$ are $\A^1$-chain homotopic. Thus $$s_1 \circ x =  s_1 \circ p \circ c_0  = \phi \circ f \circ c_0 = \phi \circ g \circ c_0 = s_2 \circ p \circ c_0 = s_2 \circ x$$ as desired. This shows that the sheaf $\sS(\sF)$ satisfies the condition (AP2).

Now we check the condition (AP1). First suppose that $U$ is a smooth, irreducible variety of dimension $\leq 2$ and we have a morphism $s: U \to \sS(\sF)$. Since $\phi: \sF \to \sS(\sF)$ is surjective, there exists a Nisnevich cover $p: V \to U$ of such that the morphism $s \circ p$ lifts to $\sF$. Thus there exists an open subscheme $U^{\prime} \subset U$ such that the morphism $s|_{U^{\prime}}$ lifts to a morphism $t: U^{\prime} \to \sF$. Applying the condition (AP1) for $\sF$, there exists a smooth projective variety $\-U$, a birational map $i: U^{\prime} \dashrightarrow \-U$ and a morphism $\-t: \-U \to \sF$ ``extending" $t$. Let $U''$ be the largest open subscheme of $U^{\prime}$ such that the rational map $i$ can be represented by a morphism $i^{\prime}: U'' \to \-U$. We know that $\-t \circ i^{\prime} \sim_0 t|_{U''}$. 

We claim that $\phi \circ \-t$ is the required morphism. We note that $i^{\prime}$ defines a rational map from $U$ to $\-U$. Let $U''' \subset U$ be the largest open subscheme of $U$ such that this rational map can be represented by a morphism $i'': U''' \to \-U$. So, we clearly have $U'' \subset U'''$. We wish to prove that for any point $x$ in $U'''$, the equality $s \circ x = \phi \circ \-s \circ i'' \circ x$ holds. 

When $x \in U^{\prime \prime}$, we already know that $\-s \circ i^{\prime} \circ x = t|_{U^{\prime \prime}} \circ x$. Composing by $\phi$, we get that $\phi \circ \-s \circ i^{\prime} \circ x = \phi \circ t|_{U^{\prime \prime}} \circ x$. But $\phi \circ t|_{U^{\prime \prime}} = s|_{U^{\prime \prime}}$. Thus we see that $\phi \circ t|_{U^{\prime \prime}} \circ x = s \circ x$. Since $i''|_{U''} = i^{\prime}|_{U''}$, we obtain the desired equality in this case. 

Thus we now assume that $x \in U''' \backslash U^{\prime \prime}$. Since $U^{\prime \prime}$ is a dense open subscheme of $U$, there exists a smooth curve $D$ over $k$ and a locally closed immersion $j: D \rightarrow U'''$ such that $j(D)$ contains $x$ and also intersects $U''$. (When $\dim(U) = 1$, $D$ will be an open subscheme of $U$.) Let $z$ be the generic point of $D$. Then we note that $j(z)$ is a point of $U''$ with residue field $k(z)$. By what we have proved in the previous paragraph, we know that $s \circ j(z) = \phi \circ \-s \circ i'' \circ j(z)$. Thus there exists an open subscheme $D^{\prime}$ of $D$ such that $j$ maps $D^{\prime}$ into $U''$ and $s \circ j|_{D^{\prime}} = \phi \circ \-s \circ i'' \circ j|_{D^{\prime}}$. Since we have already verified (AP2) for $\sS(\sF)$, we can conclude that $s \circ x = \phi \circ \-s \circ i'' \circ x$. This completes the proof of the fact that $\sS(\sF)$ is almost proper. 
\end{proof}

\begin{thm}
\label{theorem field case}
Let $\sF$ be an almost proper sheaf. Then for any field extension $K$ of $k$, $\sS(\sF)(K) = \sS^n(\sF)(K)$ for any integer $n \geq 1$. 
\end{thm}
\begin{proof}
By a base change, we may assume that $K = k$. Thus we need to prove that if $\sF$ is almost proper, we have $\sS(\sF)(k) = \sS^{n}(\sF)(k)$ for all $n \geq 1$. In view of Lemma \ref{lemma-almost-proper}, it suffices to show that $\sS(\sF)(k) = \sS^2(\sF)(k)$. 

Let $\phi$ denote the morphism of sheaves $\sF \to \sS(\sF)$. Let $x$ and $y$ be elements of $\sS(\sF)(k)$ and suppose there exists a morphism $h: \A^1_k \to \sS(\sF)$ such that $h(0) = x$ and $h(1) = y$. As $\phi$ is an epimorphism of Nisnevich sheaves, there exists an open subscheme $U$ of $\A^1_k$ such that $h|_U$ can be lifted to $\sF$, i.e. there exists a morphism $h^{\prime}: U \to \sF$ such that $\phi \circ h^{\prime} = h|_{U}$.  By (AP1), there exists a morphism $\-h: \P^1_k \to \sF$  which ``extends $h^{\prime}$ on points'', i.e. if $i: U \rightarrow \P^1_k$ is the composition $U  \hookrightarrow \A^1_k  \hookrightarrow \P^1_k$, for all points $x$ of $U$, we have $\-h \circ i \circ x = h^{\prime} \circ x$. (Here we identify $\A^1_k$ with the open subscheme of $\P^1_k = Proj(k[T_0,T_1])$ given by $T_1 \neq 0$.) Applying this to the generic point of $U$, we see that there exists an open subscheme $U^{\prime}$ of $U$ such that $\-h \circ i|_{U^{\prime}} = h^{\prime}|_{U^{\prime}}$. Since $i$ is just the inclusion, we may write this as $\-h|_{U'} = h'
|_{U'}$. 

Consider the morphism $\~h:=\phi \circ \-h|_{\A^1_k}: \A^1_k \to \sS(\sF)$. We have the equalities $\~h|_{U^{\prime}} = \phi \circ h^{\prime}|_{U^{\prime}} = h|_{U^{\prime}}$. Thus, by (AP2), we see that $\~h \sim_0 h$. But this means that $x = (\phi \circ \-h)((0:1))$ and $y = (\phi \circ \-h)((1:1))$. In other words, $x$ and $y$ are images under $\phi$ of two $\A^1$-homotopic morphisms from $\Spec(k)$ into $\sF$. But by the definition of the functor $\sS$, this implies that $x = y$.   
\end{proof}

\begin{cor} \label{corollary field case}
If $\sF$ is an almost proper sheaf of sets, the canonical surjection $\sS(F)(K) \to \pi_0^{\A^1}(K)$ is a bijection, for any finitely generated field extension $K/k$.  In particular, this holds for any proper scheme $X$ of finite type over $k$.
\end{cor}
\begin{proof}
Theorem \ref{theorem field case} and Theorem \ref{theorem lim S^n} together imply that 
$$\sS(\sF)(K) \to \sL(\sF)(K)$$
is a bijection for all finitely generated field extensions $K$ of $k$.  The result now follows by Remark \ref{remark factorization through pi_0-a1}.
\end{proof}

This completes the proof of Theorem \ref{Intro-theorem field case} stated in the introduction.

\subsection{The conjectures of Morel and of Asok-Morel for non-uniruled surfaces}
\label{subsection non-uniruled}

In this subsection, we show that Conjectures \ref{a1invariance} and \ref{conjecture Asok-Morel} hold for a non-uniruled surface over a field $k$.  In the case of non-uniruled surfaces, it turns out that the condition in Lemma \ref{lemma strategy} is true even when $\dim(U) >0$. The key observation in this case is that the problem can be reduced to $1$-dimensional schemes.  We begin by recalling the definition of a $\P^1$-bundle from \cite[Chapter II, Definition 2.5]{Kollar} which will used in the proof.

\begin{defn}
\label{definition P1-bundle}
Let $X$ be a scheme over $k$.  We say that $f:Y \to X$ is a $\P^1$-bundle if $f$ is a smooth, proper morphism such that for every point $x \in X$, the fibre $f^{-1}(x)$ is a rational curve.
\end{defn}

The following property of $\P^1$-bundles seems to be well-known to experts and is mentioned without proof in \cite[Chapter II, Definition 2.5]{Kollar}.  We include a proof here for the sake of completeness.

\begin{lemma}
\label{lemma P1-bundles}
Let $\pi: E \to B$ be a smooth, projective morphism of varieties over $k$ such that for any point $b \in B$, the fibre $E_b$ is isomorphic to $\P^1_b$. Then $\pi$ is an \'etale-locally trivial fibre bundle.  
\end{lemma}
\begin{proof}
First we consider the case that $\pi$ admits a section $\tau: B \to E$. In this case, $E$ is isomorphic to $\P(\sE)$ where $\sE$ is a rank $2$ vector bundle over $B$. For the proof of this statement, we refer to \cite[Chapter V, Lemma 2.1, Proposition 2.2]{Hartshorne}. This result is proved there under the assumption that $B$ is a smooth curve, but the proof applies in general. 

For the general case, we use the fact that the smooth, surjective morphism $\pi$ has an \'etale-local section, i.e. there exists an \'etale cover $V \to B$ and a $B$-morphism $V \to E$. Then $E_V:= E \times_B V \to V$ is also a smooth, projective morphism with fibres isomorphic to $\P^1$. However, the $B$-morphism $E \to V$ induces a section $V \to E_V$. Thus, by the special case given above, $E_V \to V$ is a Zariski-locally trivial $\P^1$-fibre bundle. It follows that $\pi$ is an \'etale-locally trivial $\P^1$-fibre bundle over $B$. 
\end{proof}

\begin{prop}
\label{proposition A1-invariance for curves}
Let $X$ be a reduced, separated $1$-dimensional proper scheme (possibly singular) over $k$. Then $\sS(X) = \sS^2(X)$. 
\end{prop}
\begin{proof}
Let $U$ be a smooth Henselian local scheme over $k$. We will show that $\sS(X)(U) = \sS^2(X)(U)$. 

If $X$ is smooth, it is a disjoint union of irreducible, smooth, projective curves $\{C_i\}_{i\in I}$ where $I$ is a finite set. We need to verify that $\sS(C_i)(U) = \sS^2(C_i)(U)$ for every $i \in I$. If $C_i$ is not geometrically rational (i.e. rational over the algebraic closure $\-k$ of $k$), it is $\A^1$-rigid (see Definition \ref{definition A1 rigid}) and thus $\sS(C_i) = C_i$. So, in this case, $\sS(C_i)(U) = \sS^2(C_i)(U)$ as desired. If $C_i$ is geometrically rational, it is isomorphic to a smooth plane conic. Any morphism $f: U \to C_i$ factors as $$U \overset{\Gamma_f}{\to} (C_i)_U:= C_i \times_{\Spec(k)} U \to C_i$$ where $\Gamma_f$ is the graph of $f$. If the morphism $(C_i)_U \to U$ has no section, then clearly $$\sS(C_i)(U) = \sS^2(C_i)(U) = \emptyset$$ which verifies our claim in this case. On the other hand, if there exists a section for the morphism $(C_i)_U \to U$, then $(C_i)_U \to U$ is a smooth, projective morphism such that the fibre over any point $u$ of $U$ is isomorphic to $\P^1_u$. (Recall that a conic is isomorphic to $\P^1$ if and only if it admits a rational point over the base field.) By Lemma \ref{lemma P1-bundles}, the morphism $(C_i)_U \to U$ is an \'etale-locally trivial $\P^1$-bundle. As $U$ is a Henselian local scheme, it is a trivial $\P^1$-fibre bundle. Thus, in this case 
\[
\sS(C_i)(U) = \sS^2(C_i)(U) = \ast
\]
\noindent as desired. Thus we have proved the lemma in case $X$ is smooth.

Suppose now that $X$ is not smooth.  For any smooth irreducible scheme $U$ over $k$, any $\A^1$-homotopy or $\A^1$-ghost homotopy of $U$ in $X$ must factor through the normalization of $X$ if its image is dense and the result follows from the smooth case.  If the image of an $\A^1$-homotopy or $\A^1$-ghost homotopy of $U$ is not dense in $X$, then it is a single point and the result follows immediately.
\end{proof}

\begin{thm}
\label{theorem non-uniruled}
Let $X$ be a proper, non-uniruled surface over $k$. Then we have $\sS(X) = \sS^2(X)$. 
\end{thm}

\begin{proof}
Since $X$ is not uniruled, we know (see \cite{Kollar}, Chapter IV, 1.3) that for any variety $Z$ over $k$, if we have a rational map $\P^1 \times Z \dashrightarrow X$, then either
\begin{itemize}
\item[(1)] this rational map $\P^1 \times Z \dashrightarrow X$ is not dominant, or 
\item[(2)] for every $z \in Z$, the induced map $\P^1_z \dashrightarrow X$ is constant.
\end{itemize}
It follows that if, for an essentially smooth scheme $U$ over $k$, we have an $\A^1$-homotopy $U \times \A^1 \to X$ , then it is either the constant homotopy or it factors through a $1$-dimensional, reduced, closed subscheme of $X$. We will make use of this observation in the following discussion. 

Let $U$ be a smooth, Henselian, local scheme over $k$. We will show that $\sS(X)(U) \to \sS(X)(U \times \A^1)$ is a bijection. If $\Dim(U) = 0$, this has already been proved in Theorem \ref{theorem field case}. Thus we now assume that $\Dim(U) \geq 1$. 

We will use the notation of Definition \ref{definition ghost homotopy} in the following arguments. So, let
\[
\sH:=\left(V \to \A^1_U, W \to V \times_{\A^1_U} V, h, h^W \right)
\]
be an $\A^1$-ghost homotopy of $U$ in $X$ connecting morphisms $t_1,t_2: U \to X$. As discussed before $\sH$ determines an $\A^1$-homotopy $\~h: \A^1_U \to \sS(X)$. We will show that either $\~h$ lifts to an $\A^1$-homotopy of $U$ in $X$ or that the given $\A^1$-ghost homotopy factors through a $1$-dimensional, reduced, closed subscheme of $X$. 

We write $V = \coprod_{i \in I} V_i$ for some indexing set $I$ where each $V_i$ is irreducible. We will denote the morphism $h|_{V_i}$ by $h_i$. Also, we write 
\[
W  = \coprod_{i,j \in I} \left(\coprod_{l \in K_{ij}} W^{ij}_l\right)
\]
\noindent for indexing sets $K_{ij}$ depending on pairs $i,j \in I$, so that each $W^{ij}_l$ is irreducible and $\left(\coprod_{l \in K_{ij}} W^{ij}_l\right) \to V_i \times_{\A^1_U} V_j$ is a Nisnevich cover. For every pair $i,j \in I$ and $l \in K_{ij}$, we have the sequence of morphisms $$W^{ij}_l \to V_i \times_{\A^1_U} V_j \to V_i \stackrel{h_i}{\to} X.$$ In the following argument, we will use $h_i|_{W^{ij}_l}$ (resp.  $h_i|_{V_i \times_{\A^1_U} V_j}$) to denote the composition of the morphisms in this sequence starting from $W^{ij}_l$ (resp. $V_i \times_{\A^1_U} V_j$). 

As we observed, the fact that $X$ is non-uniruled implies that the restrictions of $h^W$ to $W^{ij}_l$ for various $i,j \in I$ and $l \in K_{ij}$ are either constant $\A^1$-chain homotopies or they factor through some $1$-dimensional, reduced closed subscheme of $X$. If the restrictions of $h^W$ to all the $W^{ij}_k$ are constant for all $i,j \in I$ and $l \in K_{ij}$, then we have $h_i|_{W^{ij}_l} = h_j|_{W^{ij}_l}$ for all choices of $i,j$ and $l$. Since the morphism $\left(\coprod_{l \in K_{ij}} W^{ij}_l\right) \to V_i \times_{\A^1_U} V_j$ is an epimorphism, this implies that $h_i|_{V_i \times_{\A^1_U} V_j} =  h_j|_{V_i \times_{\A^1_U} V_j}$. But this means that the  morphisms $h_i$ give rise to a morphism $\A^1_U \to X$ which lifts $\~h: \A^1_U \to \sS(X)$. By the definition of $\sS(X)$, this proves $t_1 = t_2$.

Thus we now assume that for some fixed $i,j \in I$ and some fixed $l \in K_{ij}$, the restriction of $h^W$ to $W^{ij}_l$ is a non-constant $\A^1$-chain homotopy.  Hence, (2) above cannot hold.  We claim that in this case all the morphisms $h_i$ for $i \in I$ as well as the $\A^1$-chain homotopy $h^W$ factor through some $1$-dimensional, reduced, closed subscheme of $D$ of $X$. Once this is proved, the result will follow from Proposition \ref{proposition A1-invariance for curves}.  Indeed, then we see that the $t_1,t_2: U \to X$ must factor through $D$ and, being connected by an $\A^1$-ghost homotopy within $D$, must also be connected by an $\A^1$-chain homotopy in $D$. It now remains to prove the existence of $D$. 

The restriction of $h^W$ to $W^{ij}_l$ must factor through some $1$-dimensional, closed subscheme of $C^{ij}_l \subset X$, by (1) above. In particular, the morphisms $h_i|_{W^{ij}_{l}}$ and $h_j|_{W^{ij}_{l}}$ also factor through $C^{ij}_l \hookrightarrow X$. As $V_i$ and $V_j$ are assumed to be irreducible and since the morphisms $W^{ij}_l \to V_i$ and $W^{ij}_l \to V_j$ are dominant, both $h_i$ and $h_j$ also factor through $C^{ij}_l \hookrightarrow X$. 

Now we consider any other index $m \in K_{ij}$. If the restriction of $h^W$ to $W^{ij}_m$ is a non-constant $\A^1$-chain homotopy, the argument in the previous paragraph shows that this $\A^1$-chain homotopy factors through some $1$-dimensional, reduced, closed subscheme $C^{ij}_{m}$. On the other hand, if this $\A^1$-chain homotopy is constant, it must factor through the variety $C^{ij}_l$ obtained above and thus we may define $C^{ij}_m:= C^{ij}_l$. Taking the union $C^{ij}:= \bigcup_{m \in K_{ij}} C^{ij}_m$, we see that  $h_i$, $h_j$ and the restriction of $h^W$ to $\coprod_{m \in K_{ij}} W^{ij}_m$ factor through $C^{ij}$. 

Now let $i' \in I$ be any index. We claim that $h_{i'}$ factors through some $1$-dimensional, reduced, closed subscheme $C^i$ of $X$. For $i' = i$ or $j$, we define $C^{i'}:= C^{ij}$. For $i' \neq i,j$, we prove the claim by contradiction. Suppose that $h_{i'}$ does not factor through any $1$-dimensional, reduced, closed subscheme of $X$. First note that the scheme $V_i \times_{\A^1_U} V_{i^{\prime}}$ is non-empty since $\A^1_U$ is irreducible and $V_i, V_{i^{\prime}}$ are \'etale over $\A^1_U$. Thus the scheme $\coprod_{m \in K_{ii'}} W^{ii'}_m$ is non-empty. The restriction of $h^W$ to $\coprod_{m \in K_{ii'}} W^{ii'}_m$ cannot be the constant $\A^1$-chain homotopy since $h_i$ factors through a $1$-dimensional subscheme of $X$ while $h_{i'}$ does not. But then, by the argument in the previous paragraph, the fact that this $\A^1$-chain homotopy is nonconstant implies that $h_{i'}$ factors through some $1$-dimensional, reduced, closed subscheme of $C^{ii'}$ of $X$. This gives us a contradiction. Thus for every index $i' \in I$, the morphism $h_{i'}$ factors through some $1$-dimensional, reduced, closed subscheme $C^{i'}$ of $X$.  

Finally, we wish to prove that for any two indices $i', j' \in I$, the restriction of $h^W$ to $\coprod_{m \in K_{i'j'}} W^{i'j'}_m$ factors through some $1$-dimensional, reduced, closed subscheme of $X$. If this restriction is a non-constant $\A^1$-homotopy, the above argument gives us a $1$-dimensional, reduced, closed subscheme $C^{i'j'}$ of $X$ through which this homotopy must factor. If it is the constant homotopy, we note that both $h_{i'}$ and $h_{j'}$ factor through $C^{i'} \cup C^{j'}$. Thus, the restriction of $h^W$ to $\coprod_{m \in K_{i'j'}} W^{i'j'}_m$ must also factor through this subscheme. 

Now we may define $D$ to be the scheme $(\bigcup_{i,j \in I} C^{ij}) \cup (\bigcup_{i \in I} C^i)$. This completes the proof of the theorem.
\end{proof}

\begin{cor}
Conjectures \ref{a1invariance} and \ref{conjecture Asok-Morel} hold for a proper, non-uniruled surface over a field.
\end{cor}

\begin{rmk}
One can prove that $\sS(X) \simeq \sS^2(X)$, when $X$ is a rational surface (or in fact, any rational variety) or a $\P^1$-bundle over a curve.  However, for a ruled surface $X$ (over an algebraically closed base field of characteristic 0), one can only prove that $\sS(X)(U) \simeq \sS^2(X)(U)$, where $U$ is the Spec of a Henselian discrete valuation ring.  In fact, the conclusion of Theorem \ref{theorem non-uniruled} is not true for ruled surfaces (even over an algebraically closed base field of characteristic $0$).  Indeed, if $C$ is a curve of genus $>0$ and $X$ is the blowup of $\P^1 \times C$ at a single closed point, then it is possible to show that $\sS(X)(U) \neq \sS^2(X)(U)$, for a two-dimensional smooth Henselian local scheme $U$. This shows that $\sS(X) \neq \sS^2(X)$ in general, even for smooth projective surfaces.  Surprisingly, for such ruled surfaces $X$, one can prove that $\sS^2(X) \simeq \sS^3(X)$.  For details, see \cite{Sawant-thesis}.   However, our proof for these results involves a lot more effort and the techniques used there substantially differ from the ones used in this paper.  We do not include it here since we can demonstrate this phenomenon in a much easier manner in higher dimensions (see Section \ref{section counterexamples}).  The details about $\A^1$-connected components of ruled surfaces will appear in a forthcoming paper \cite{BS}.
\end{rmk}

\begin{rmk}
Although the method of showing $\sS(X) \simeq \sS^2(X)$ is insufficient to address the question of the $\A^1$-invariance of $\pi_0^{\A^1}$ for an arbitrary smooth projective variety $X$, it is interesting to examine the question of whether, for any smooth projective variety $X$, the sequence of sheaves $(\sS^n(X))_{n \geq 1}$ stabilizes at some finite value of $n$.  If it does, one may ask whether it stabilizes to $\pi_0^{\A^1}(X)$, since that is what is predicted by Morel's conjecture (see Corollary \ref{a1-invariance-candidate}).  
\end{rmk}

\section{Examples of schemes for which the conjectures of Asok-Morel fail to hold}
\label{section counterexamples}

\subsection{Examples of schemes whose \texorpdfstring{$Sing_*$}{Sing*} is not \texorpdfstring{$\A^1$}{A1}-local}
The aim of this subsection is to construct an example of a smooth, projective variety $X$ over $\C$ such that:
\begin{itemize}
\item[(i)] $\sS(X) \neq \sS^2(X)$,
\item[(ii)] $\sS(X) \to \pi_0^{\A^1}(X)$ is not a monomorphism, and
\item[(iii)] $Sing_*(X)$ is not $\A^1$-local.
\end{itemize}

We will use a well-known characterization of Nisnevich sheaves, which we will recall here for the sake of convenience.

For any scheme $U$, an \emph{elementary Nisnevich cover} of $U$ consists of two morphisms $p_1: V_1 \to U$ and $p_2: V_2 \to U$ such that:
\begin{itemize}
 \item[(a)] $p_1$ is an open immersion.
 \item[(b)] $p_2$ is an \'etale morphism and its restriction to $p_2^{-1}(U \backslash p_1(V_1))$ is an isomorphism onto $U \backslash p_1(V_1)$. 
\end{itemize}
Then a presheaf of sets $\sF$ on $Sm/k$ is a sheaf in Nisnevich topology if and only if the morphism $$\sF(U) \to \sF(V_1) \times_{\sF(V_1 \times_U V_2)} \sF(V_2)$$ is an isomorphism, for all elementary Nisnevich covers $\{ V_1, V_2\}$ of $U$. If $\sF$ is a Nisnevich sheaf, the fact that this morphism is an isomorphism is an immediate consequence of the definition of a sheaf. The converse (see \cite[\textsection 3, Proposition 1.4, p.96]{Morel-Voevodsky} for a proof) is useful since it simplifies the criterion for checking whether sections $s_i \in \sF(V_i)$ determine a section  $s \in \sF(U)$. Indeed, suppose $pr_1$ and $pr_2$ are the two projections $V \times_U V \to V$. Then we do not need to check that the two elements $pr_1^*(s_1)$ and $pr_2^*(s_2)$ of $\sF(V \times_U V)$ are identical. We merely need to check that the images of $s_1$ and $s_2$ under the maps $\sF(V_1) \to \sF(V_1 \times_U V_2)$ and $\sF(V_2) \to \sF(V_1 \times_U V_2)$ respectively are identical.  

Now suppose $\sF$ is a sheaf of sets. Applying the above criterion to the sheaf $\sS(\sF)$, we see that when we work with elementary Nisnevich covers, we can construct morphisms $\A^1_U \to \sS(\sF)$ with only part of the data that is required for an $\A^1$-ghost homotopy of $U$ in $\sF$. 

\begin{lem}
\label{lemma Sing_* not A1-local}
Let $\sF$ be a sheaf of sets over $Sm/k$. Let $U$ be a smooth scheme over $k$ and let $f,g: U \to \sF$ be two morphisms. Suppose that we are given data of the form
\[
 (\{p_i: V_i \to \A^1_U\}_{i \in \{1,2\}}, \{\sigma_0,\sigma_1\}, \{h_1, h_2\}, h^W)
\]
where:
\begin{itemize}
 \item The two morphisms $\{p_i: V_i \to \A^1_U\}_{i=1,2}$ constitute an elementary Nisnevich cover.
 \item For $i \in \{0,1\}$, $\sigma_i$ is a morphism $U \to V_1 \coprod V_2$ such that $(p_1 \coprod p_2) \circ \sigma_i: U \to U \times \A^1$ is the closed embedding $U \times \{i\} \hookrightarrow U \times \A^1$,
 \item For $i \in \{1,2\}$, $h_i$ denotes a morphism $V_i \to \sF$ such that $(h_1 \coprod h_2) \circ \sigma_0 = f$ and $(h_1 \coprod h_2) \circ \sigma_1 = g$. 
 \item Let $W := V_1 \times_{\A^1_U}V_2$ and let $pr_i: W \to V_i$ denote the projection morphisms. Then $h^W = (h_1, \ldots, h_n)$ is a $\A^1$-chain homotopy connecting the two morphisms $h_i \circ pr_i: W  \to \sF$. 
\end{itemize}
Then, $f$ and $g$ map to the same element under the map $\sF(U) \to \pi_0^{\A^1}(\sF)(U)$.  Moreover, $Sing_*(\sF)$ is not $\A^1$-local.
\end{lem}

\begin{proof}
Choose a simplicially fibrant replacement $Sing_*(\sF) \to \sX$.  We denote by $H_i$ the composition $V_i \overset{h_i}{\to} \sF \to Sing_*(\sF) \to \sX$, for $i = 1,2$. According to the discussion preceding the statement of this lemma, the given data induces a morphism $\psi: \A^1_U \to \sS(\sF) = \pi_0^s(\sX)$. 

The $\A^1$-homotopy $h^{W}$  connects $h_i|_W$ for $i = 1,2$. Thus, by the definition of $Sing_*(\sF)$, it gives rise to a simplicial homotopy $h^s: \Delta^1 \to \sX(W)$ connecting $H_i$ for $i=1,2$. Thus we have the following diagram (where $*$ denotes the point sheaf)
\[
\xymatrix{
{*} \ar[r] \ar[d] & \sX(V_2) \ar[d] \\
\Delta^1 \ar[r] \ar@{-->}[ur]^{\exists} & \sX(W) 
}
\]
where the upper horizontal arrow maps $*$ to $H_2$.  Since $W = V_2 \backslash (U \times \pi^{-1}({0})) \to V_2$ is a cofibration, the morphism $\sX(V_2) \to \sX(W)$ of simplicial sets is a fibration. Thus the dotted diagonal lift exists in the above diagram. 

Thus we see that we can find a morphism $H^{\prime}: V_2 \to \sX$ such that:
\begin{itemize}
\item[(1)] $h^{\prime}$ and $h_2$ induce the same morphism $V_2 \to \pi_0^s(\sX)$ (since they are simplicially homotopic in $\sX(V_2)$), and
\item[(2)] $H^{\prime}|_W = H_1|_{W}$. 
\end{itemize}

Since $\sX$ is a sheaf, the following diagram is cartesian:
\[
\xymatrix{
\sX(\A^1_U) \ar[r] \ar[d] & \sX(V_2) \ar[d] \\
\sX(V_1) \ar[r] & \sX(W)\text{.}
}
\]
Thus, $H_1$ and $H^{\prime}$ can be glued to give an element of $\sX(\A^1_U)$ which we can think of as a morphism $\~\psi: \A^1_U \to \sX$. Clearly, this lifts $\psi$, as can be checked by restriction to the Nisnevich cover $\{V_1, V_2\}$.  

It is easily seen that the morphisms that $\~\psi$ connects are simplicially chain homotopic to the morphisms $\~f$ and $\~g$. Thus we see that the morphisms $\~f$ and $\~g$ are simplicially chain homotopic. From the explicit construction of the $\A^1$-fibrant replacement functor $L_{\A^1}$ given in \cite{Morel-Voevodsky}, we see that the two morphisms $U \to L_{\A^1}(\sF)$ obtained by composing $f$ and $g$ with the canonical morphism $\sF \to L_{\A^1}(\sF)$ are simplicially homotopic. Therefore, $f$ and $g$ map to the same element of $\pi_0^{\A^1}(\sF)(U)$.  Consequently, $\sX$ (and hence, $Sing_*(\sF)$) is not $\A^1$-local.
\end{proof} 

\begin{convention}
We will refer to the data of the form $$(\{p_i: V_i \to \A^1_U\}_{i \in \{1,2\}}, \{\sigma_0,\sigma_1\}, \{h_1, h_2\}, h^W)$$ as in Lemma \ref{lemma Sing_* not A1-local} by an \emph{elementary $\A^1$-ghost homotopy}. 
\end{convention}

We now proceed to the construction of the schemes satisfying properties (i)-(iii) stated at the beginning of the section.

\begin{construction}
\label{notation example 1}
To clarify the idea behind this construction, we first construct a non-proper, singular variety satisfying (i)-(iii). Later (in Construction \ref{notation example 2}), we will modify this example suitably to create an example of a smooth, projective variety over $\C$ which will also satisfy (i)-(iii).
\begin{itemize}
\item[(1)] Let $\lambda_i \in \C \backslash \{0\}$ for $i = 1,2,3$. Let $E \subset \A^2_{\C}$ be the affine curve cut out by the polynomial $y^2 - \prod_i (x-\lambda_i)$. Let $\pi: E \to \A^1_{\C}$ denote the projection onto the $x$-axis.
\item[(2)] Let $S_1$ and $S_2$ be the closed subschemes of $\A^3_{\C}$ cut out by the polynomials $y^2 = t^2 \prod_i (x - \lambda_i)$ and $y^2 = \prod_i (x-\lambda_i)$ respectively. Let $f: S_2 \to S_1$ be the morphism corresponding to the homomorphism of coordinate rings given by $x \mapsto x$, $y \mapsto yt$ and $t \mapsto t$. 
\item[(3)] Let $\alpha_0$ be a square-root of $-\lambda_1\lambda_2\lambda_3$ so that the point $(0,\alpha_0)$ is mapped to the point $x=0$ under the projection $\pi$. Let $p_0$, $p_1$ and $q$ denote the points $(0,0,0)$, $(0,\alpha_0,1)$ and $(1,0,0)$ of $S_1$. 
\item[(4)] Consider the morphism $\A^1_{\C} \to S_2$ given by $s \mapsto (0,\alpha_0,s)$ for $s \in \A^1_{\C}$. Composing this morphism with $f$, we obtain a morphism $\A^1_{\C} \to S_1$, the image is a closed subscheme of $S_1$ which we denote by $D$.
\end{itemize}
\end{construction}

Now let $S \subset \A^3_{\C}$ be the scheme $S_1 \backslash \{p_0\}$. We claim that $S$ satisfies condition (i) above. We claim that any morphism $h:\A^1_{\C} \to S$ such that $p_1$ lies in the image of $h$ is constant.  To see this, first observe that this is obvious if the image of $h$ is contained completely inside $D \backslash \{p_0\}$ (since any morphism from $\A^1_{\C}$ into $\A^1_{\C} \backslash \{0\}$ is a constant morphism).  If this is not the case, note that the morphism $f: S_2 = E \times \A^1_{\C} \to S_1$ is an isomorphism outside the locus of $t=0$.  Hence, $h$ induces a rational map $\A^1_{\C} \dashrightarrow E$, which can be completed to a morphism $\P^1_{\C} \to \-E$, where $\-E$ denotes a projective compactification of $E$.  This has to be a constant morphism by the Hurwitz formula, since the genus of $\-E$ is greater than that of $\P^1_{\C}$. Thus, there is no non-constant morphism $\A^1_{\C} \to S$ having $p_1$ in its image. Thus we see that $p_1$ and $q$ are not connected by an $\A^1$-
chain homotopy. However, we claim that they map to the same element in $\sS^2(S)(\C)$.

To construct an elementary $\A^1$-ghost homotopy, we first construct a Nisnevich cover of $\A^1_{\C}$. We observe that $\pi$ is \'etale except at the points where $x = \lambda_i$, $i = 1,2,3$. Now we define $\-V_1 = \A^1_{\C} \backslash \{0\}$ and $\-V_2 = \pi^{-1}(\A^1_{\C} \backslash \{\lambda_1, \lambda_2, \lambda_3\}) \subset E$. The morphism $\-V_1 \to \A^1_{\C}$ is the inclusion and the morphism $\-V_2 \to \A^1_{\C}$ is simply $\pi|_{\-V_2}$. As $\lambda_i \neq 0$ for all $i$, we see that this is an elementary Nisnevich cover of $\A^1_{\C}$. We define $\-W:= \-V_2 \backslash \pi^{-1}(0)$ and observe that $\-V_1 \times_{\A^1_{\C}} \-V_2 \cong \-W$. In order to give an elementary $\A^1$-ghost homotopy in $S$ using this Nisnevich cover, we will now define morphisms $\-h_i: \-V_i \to S$ and a morphism $h^{\-W}: \-W \times \A^1 \to S$, which is a homotopy connecting $(\pi \circ \-h_1)|_{\-W}$ and $\pi \circ \-h_2|_{\-W}$. 

We define $\-h_1: \-V_1 \to S$ by $s \mapsto (s,0,0)$ for $s \in \A^1_{\C}$ and $\-h_2: \-V_2 \to S$ by $(\alpha,\beta) \mapsto (\alpha,\beta,1)$ for $(\alpha,\beta) \in \-V_2$. To define $h^{\-W}$, we first identify $S_2$ with $E \times \A^1$ and define $h^{\-W}$ to be $f|_{\-W \times \A^1}$ (recall that $\-W \subset \-V_2 \subset E$). It is clear that the the $\C$-valued points $p_1$ and $q$ are mapped to the same element in $\sS^2(S)(\C)$. This shows that the surface $S$ satisfies the condition (i) listed above. 

We have thus obtained the data
\[
\left( \-V_1 \coprod \-V_2 \to \A^1_U, \-W \to \-V_1 \times_{\A^1_U} \-V_2, \-h_1 \coprod \-h_2, \-h^{\-W}\right), 
\]
\noindent which satisfies the hypotheses of Lemma \ref{lemma Sing_* not A1-local}, which proves that $S$ satisfies (ii) and (iii).

The following lemma is the key tool in creating a smooth example that satisfies (i) - (iii) stated at the beginning of the section:

\begin{lem}
\label{lemma rigid embedding}
Let $k$ be a field and let $S$ be an affine scheme over $k$. Then there exists a closed embedding of $S$ into a smooth scheme $T$ over $k$ such that for any field extension $L/k$, if $H: \A^1_L \dashrightarrow T$ is a rational map such that the image of $H$ meets $S$, then $H$ factors through $S \hookrightarrow T$. 
\end{lem}
\begin{proof}
Choose an embedding of $S$ into $\A^n_k$ for some $n$ and suppose that as a subvariety of $\A^n_k$, $S$ is given by the ideal $\<f_1, \ldots, f_r\> \subset k[X_1, \ldots, X_n]$. Then the polynomials $f_i$ define a morphism $f: \A^n \to \A^r$ such that $S$ is the fibre over the origin. Now let $g: C \to \A^1_k$ be an \'etale morphism such that $C$ is a curve of genus $\geq 1$ and the preimage of the origin consists of a single $k$-valued point $c_0$ of $C$. Then this gives us an \'etale morphism $g^r: C^r \to \A^r$ such that the preimage of the origin is the point $(c_0, \ldots, c_0)$. Let $T = \A^n_k \times_{f,\A^r_k, g^r} C^r$. It is clear that $T \to \A^n_k$ induces an isomorphism over $S$. We claim that $T$ is the desired scheme. 

Let $L/k$ be a field extension and let $H: \A^1_L \dashrightarrow T$ be a rational such that the image of $H$ meets the preimage of $S$ in $T$. Since $C$ has genus $\geq 1$, the Hurwitz formula implies that any rational map $\A^1_L \dashrightarrow C$ is constant. Thus, the composition of $H$ with the projection $T \to C^r$ factors through $(c_0, \ldots, c_0)$. It follows from this that $H$ factors through $S$, which is the preimage of $(c_0, \ldots, c_0)$ in $T$.   
\end{proof}

\begin{construction}
\label{notation example 2}
We now construct a smooth and projective variety $X$ over $\C$ satisfying the conditions (i)-(iii) listed at the beginning of this subsection:
\begin{itemize}
\item[(1)] Using Lemma \ref{lemma rigid embedding}, we construct a scheme $T$ such that $S_1 \subset T$ as a closed subscheme and such that for any field extension $L/\C$, if $H: \A^1_L \to T$ is a morphism such that the image of $H$ meets $S_1$, then $H$ factors through $S_1$.
\item[(2)] Let $T^c \supset T$ be a smooth, projective compactification of $T$. 
\item[(3)] Let $C$ be a smooth, projective curve over $\C$ of genus $>0$. Let $c_0$ be a $\C$-valued point of $C$. Let $R = \sO^{h}_{C,c_0}$ and let $U = \Spec(R)$. Let $u$ be the closed point of $U$. Let $\gamma: U \to C$ be the obvious morphism.
\item[(4)] Let $p_2$ by any closed point of $D \backslash \{p_0, p_1\}$ (see Construction \ref{notation example 1}, (4)). Let $X$ be the blowup of $C \times T^c$ at the points $(c_0,p_0)$ (see Construction \ref{notation example 1}, (3)) and $(c_0,p_2)$.
\item[(5)] For any $\C$ valued point $p: \Spec(\C) \to T^c$ of $T^c$, let $\theta_p$ be the induced morphism $\gamma \times p: U = U \times \Spec(\C) \to C \times T^c$. Note that $\theta_p$ is a lift of $\gamma$ to $C \times T^c$ with respect to the projection $C \times T^c \to C$.
\item[(6)] Let $\xi_{p_1}$ and $\xi_q$ denote the lifts of $\theta_{p_1}$ and $\theta_{q}$ to $X$ (which exist and are unique).
\end{itemize}
\end{construction}

We claim that $\xi_{p_1}$ and $\xi_q$ are connected by an elementary $\A^1$-ghost homotopy. To prove this, we construct an elementary $\A^1$-ghost homotopy connecting $\theta_{p_1}$ and $\theta_{q}$ which lifts to $X$. This is done by simply taking the product of the morphism $\gamma: U \to C$ with the elementary $\A^1$-ghost homotopy of $\Spec(\C)$ in $S \subset T^c$ which we constructed above. In other words, we define $V_i = U \times \-V_i$ and $W = U \times \-W$. We define $h_i := \gamma \times \-h_i: V_i = U \times \-V_i \to C \times T^c$ and define $h^W:= \gamma \times h^{\-W}: W \times \A^1 = U \times \-W \to C \times T^c$. (Here we have abused notation by viewing $\-h_i$ and $h^{\-W}$ as morphisms into $T^c$ rather than as morphisms into $S \subset T^c.$) We observe that $W \cong V_1 \times_{\A^1_U} V_2$ and we have constructed an elementary $\A^1$-ghost homotopy $\sH$ of $U$ in $C \times T^c$. Since this elementary $\A^1$-ghost homotopy factors through the complement of the the points $(c_0,p_0)$ 
and $(c_0,p_2)$, it lifts to an elementary $\A^1$-ghost homotopy $\~{\sH}$ of $U$ in $X$ and connects $\xi_{p_1}$ and $\xi_q$. 

We now claim that $\xi_{p_1}$ and $\xi_q$ are not $\A^1$-chain homotopic. For this it will suffices to show that if $\xi_{p_1}$ is $\A^1$-homotopic to any $\xi: U \to X$, then $\xi(u) = \xi_{p_1}(u)$. Projecting this $\A^1$-homotopy down to $C \times T^c$, we see that it suffices to prove that if $h$ is an $\A^1$-homotopy of $U$ in $C \times T^c$ which lifts to $X$ and if $\sigma_0 \circ h = \theta_{p_1}$, then $h$ maps the closed fibre $\A^1_{\C} \subset \A^1_U$ to $\theta_{p_1}(u)$. 

Let $\-D$ be the closure of $D$ (see Construction \ref{notation example 1}, (4)) in $C \times T^c$. We see by Construction \ref{notation example 2}, (4) that $\-D$ is the only rational curve through the point $\theta_{p_1}(u)$. We claim that any $\A^1$-homotopy $h$ of $U$ in $C \times T^c$ such that $h \circ \sigma_0 = \theta_{p_1}$ must map $\A^1_{\C} \subset \A^1_U$ into $D$. Indeed, this follows immediately from the fact that the composition of this $\A^1$-homotopy with the projection $pr_1: C \times T^c \to C$ must be the constant homotopy (since $C$ is $\A^1$-rigid). In fact, it is easy to see that $pr_1 \circ h$ must factor through $\gamma$.

Now we claim that the image of $h|_{\A^1_{\C}}$ does not contain the point $p_0$. Indeed, if it were to contain this point, since $h$ can be lifted to $X$, the preimage of $p_0$ under $h$ would be a non-empty, codimension $1$ subscheme of $\A^1_U$ which is contained in $\A^1_{\C} \subset \A^1_U$. Thus it would be equal to $\A^1_{\C}$. Since $h \circ \sigma (u) = p_1 \neq p_0$, this is impossible. Thus, we see that the image of $h|_{\A^1_{\C}}$ does not contain $p_0$. By the same argument it does not contain $p_2$.  

However, a morphism from $\A^1_{\C}$ into a rational curve which avoids at least two points on the rational curve must be a constant (since $\G_m$ is $\A^1$-rigid). Thus $h|_{\A^1_{\C}}$ is a constant. This completes our proof of the claim that $\xi_{q}$ and $\xi_{p_1}$ are not $\A^1$-chain homotopic. Thus, we have now proved that $X$ satisfies the property (i) listed above. Lemma \ref{lemma Sing_* not A1-local} implies that it also satisfies property (ii). If $\Sing_*(X)$ were $\A^1$-local, the morphism $\sS(X):= \pi_0^s(Sing_*(X)) \to \pi_0^{\A^1}(X)$ would be an isomorphism. Thus $X$ also satisfies property (iii). 

\begin{rmk}
Note that this method can be used to prove that $Sing_*(X)$ is not $\A^1$-local for any scheme $X$ for which the following holds: there exists a smooth Henselian local scheme $U$ and two morphisms $f, g:U \to X$, which are not $\A^1$-chain homotopic and are such that there exists an $\A^1$-ghost homotopy connecting $f$ to $g$ that is defined on an elementary Nisnevich cover of $\A^1_U$.  Such a scheme $X$ is obviously a counterexample to $\sS(X) \simeq \sS^2(X)$ and hence, also a counterexample to Conjecture \ref{conjecture Asok-Morel} of Asok and Morel.
\end{rmk} 
 
\subsection{Example showing that \texorpdfstring{$\pi_0^{\A^1}$}{pi0A1} is not a birational invariant}

We end this paper by noting a counterexample to the birationality of $\pi_0^{\A^1}$ of smooth proper schemes over a field $k$.

\begin{defn}
\label{definition A1 rigid}
A scheme $X \in {\mathcal Sm}_k$ is said to be \emph{$\A^1$-rigid}, if for every $U \in {\mathcal Sm}_k$, the map
\[
X(U) \longrightarrow X(U \times \A^1)
\]
induced by the projection map $U \times \A^1 \to U$ is a bijection.
\end{defn}

Since schemes are simplicially fibrant objects, by \cite[\textsection 2, Lemma 3.19]{Morel-Voevodsky}, it follows that $\A^1$-rigid schemes are $\A^1$-fibrant objects.  Thus, if $X \in Sm/k$ is $\A^1$-rigid, then for any $U \in Sm/k$ the canonical map $X(U) \to \Hom_{\mathcal H(k)}(U,X)$ is a bijection.  Consequently, the canonical map $X \to \pi_0^{\A^1}(X)$ is an isomorphism of Nisnevich sheaves of sets.

\begin{ex} \label{example pi_0 not birational}
Let $X$ be an abelian variety of dimension at least $2$.  Since $X$ is an $\A^1$-rigid scheme, $\pi_0^{A^1}(X) \simeq X$, as noted above.  Now, blow up $X$ at a point to get another scheme $Y$.  Then the blow-up map $Y \to X$ is a birational proper map which induces a map $\pi_0^{\A^1}(Y) \to \pi_0^{\A^1}(X)$.  If $\pi_0^{\A^1}$ is a birational invariant, then $\pi_0^{\A^1}(Y) \to \pi_0^{\A^1}(X)$ must be an isomorphism. However, $Y\to \pi_0^{\A^1}(Y)$ is an epimorphism, implying that the map $Y\to X$ is an epimorphism (in Nisnevich topology).  But a blow-up cannot be an epimorphism in Nisnevich topology, unless it is an isomorphism.  Indeed, if $Y \to X$ is an epimorphism, it would admit a Nisnevich local section, which in turn, would imply that it is an isomorphism, $Y \to X$ being a birational proper morphism.  This is a contradiction.  Consequently, $\pi_0^{\A^1}(X)$ cannot be a birational invariant of smooth proper schemes.  
\end{ex}

\end{document}